\documentclass[12pt]{article}
\usepackage{amsmath,amssymb,amsfonts,amsthm,epsfig,url,xspace}

\newtheorem{theorem}{Theorem}

\newtheorem{lemma}[theorem]{Lemma}

\newtheorem{corollary}[theorem]{Corollary}
\theoremstyle{definition}
\theoremstyle{definition}

\newcommand{\R}{\mathbb{R}}

\newcommand{\BP}{\text{BP}}
\newcommand{\eps}{\epsilon}
\newcommand{\ep}{\varepsilon}
\newcommand{\old}[1]{}

\newcommand{\pati}{\frac{\partial}{\partial t_i}}
\newcommand{\pa}{\partial}

\title{Branched Polymers}
\author{Richard Kenyon\thanks{Mathematics Department, Brown University, Providence RI 02912, USA; richard.kenyon@brown.edu. 
Research supported by NSERC. Research of both authors
began at a workshop of the Aspen Institute for Physics.}
and Peter Winkler
\thanks{Department of Mathematics, Dartmouth,
Hanover NH 03755-3551, USA; peter.winkler@dartmouth.edu.  Research
supported by NSF grant DMS-0600876.}}
\date{}
\begin{document}

\maketitle

\begin{abstract}
Building on and from the work of Brydges and Imbrie, we give an elementary
calculation of the volume of the space of branched polymers of order
$n$ in the plane and in 3-space.  Our development reveals some more
general identities, and allows exact random sampling. In particular
we show that a random $3$-dimensional branched polymer
of order $n$ has diameter of order $\sqrt{n}$.
\end{abstract}

\section{Introduction}

A {\bf branched polymer of order $n$} in $\R^d$---or just ``polymer'' for short---is
a connected set of $n$ labeled unit spheres with nonoverlapping interiors.
We will assume that the sphere labeled $1$ is centered at the origin.
See~Figure \ref{fig:bp1} for an example in the plane.

Intended as a model in chemistry or biology, branched polymers are often modeled,
in turn, by lattice animals (trees on a grid); see, e.g., \cite{Bu,F,KS,Lu,Va,Vu}.
However, we will see that continuum polymers turn out to be in some respects more tractable.

The set of polymers can be parametrized locally by the spherical angles
of the vectors connecting adjacent sphere centers.  In these coordinates
Brydges and Imbrie \cite{BI} showed that the space $B^d(n)$ of
polymers of order $n$ has total volume $(n{-}1)!(2\pi)^{n-1}$ for $d=2$ and
$n^{n-1}(2\pi)^{n-1}$ for $d=3$.  Their proof uses nonconstructive
techniques such as equivariant cohomology and localization.

We give here an elementary proof, together with some generalizations and
an algorithm for exact random sampling of polymers.
In the planar case our algorithm has the added feature of being
inductive, in the sense that a uniformly random polymer of order
$n$ is constructed from one of order $n{-}1$.

Although not explicit in their paper, the proof in \cite{BI} in fact shows 
that in the planar case the volume of the configuration space is unchanged
when the radii of the individual disks are different. We use this fact in an
essential way in our constructions.

\begin{figure}[htbp]
\epsfxsize220pt
$$\epsfbox{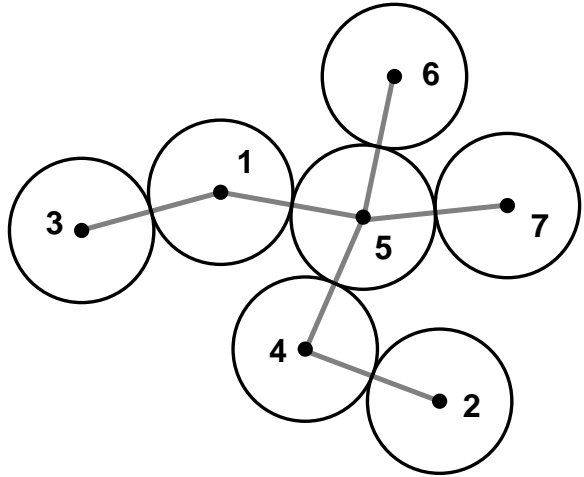}$$
\caption{A branched polymer in the plane}\label{fig:bp1}
\end{figure}

\section{The Planar Case}

Let us observe first that $(n{-}1)!(2\pi)^{n-1}$ is also the volume of the space
of ``crossing worms''---that is, strings of labeled touching disks, beginning
with disk 1 centered at the origin, but now with no constraint that disks
may not overlap.  See Figure~\ref{fig:bp2} below for an example.  Fixing the
order of disks 2 through $n$ in the crossing worm yields an ordinary unit-step
walk in the plane of $n{-}1$ steps.

\begin{figure}[htbp]
\epsfxsize220pt
$$\epsfbox{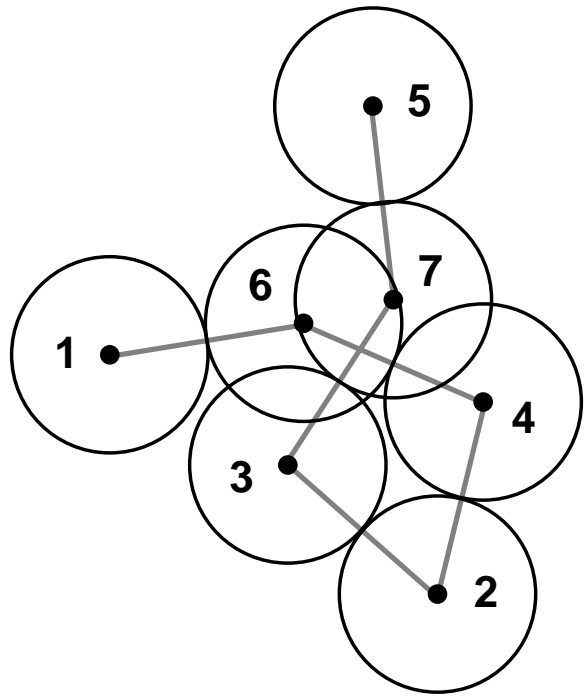}$$
\caption{A crossing worm}\label{fig:bp2}
\end{figure}

Yet another space of volume $(n{-}1)!(2\pi)^{n-1}$ is the space of
``crossing inductive trees'', one of which is illustrated in Figure~\ref{fig:bp3}.
A crossing inductive tree is a tree of $n$ touching labeled disks
with overlapping permitted, but required to satisfy the condition
that for each $k<n$, disks $1,\dots,k$ must also form a tree.  In other words,
the vertex labels increase from the root $1$. We will see that this
space is in fact a certain limiting case of the space of polymers.

\begin{figure}[htbp]
\epsfxsize220pt
$$\epsfbox{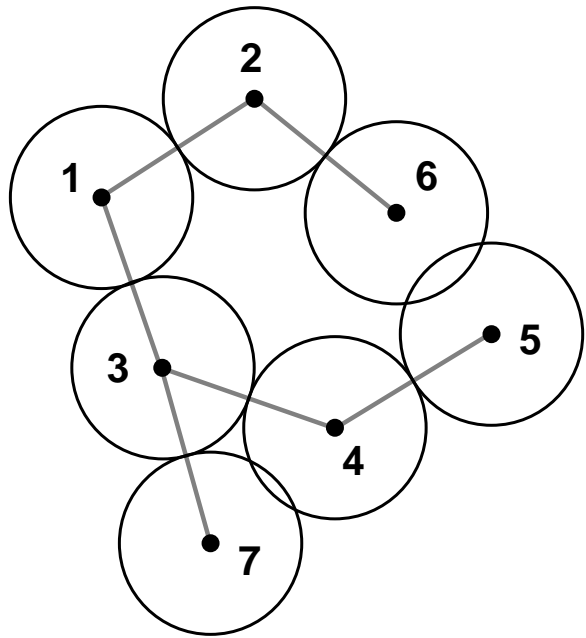}$$
\caption{A crossing inductive tree}\label{fig:bp3}
\end{figure}

\subsection{Coordinates}

In the volume calculation we will need to consider polymers made of disks
of arbitrary radius. Let $r_i\in(0,\infty)$ be the radius of the $i$th disk and
$R=\{r_1,\dots,r_n\}$ be the vector of radii.  Given a polymer $X=X(R)$,
define a graph $G(X)$ with a vertex for each disk of $X$ and an edge between
vertices whenever the corresponding disks are adjacent.  Almost surely $G(X)$
is a tree, that is, has no cycles.  When $G(X)$ is a tree, we root $G(X)$ at
the origin, and direct each edge away from the origin.  This allows us to assign
an ``absolute'' angle (taken counterclockwise relative to the $X$-axis)
to each edge.  Let $e_1,\dots,e_{n-1}$ be the edges (chosen in some
order) and $\theta_1,\dots,\theta_{n-1}$ the corresponding angles.

For a given combinatorial tree $T$, the set of polymers $X=X(R,T)$
with graph $G(X)=T$ can thus be identified with a subset of $[0,2\pi)^{n-1}$.
Call this set $\BP_R(T)$.
The boundary of $\BP_R(T)$ corresponds to polymers having at least one cycle;
the corresponding plane graphs $G(X)$ are obtained by adding one or more edges to $T$.
Indeed, the boundary of $\BP_R(T)$ is piecewise analytic and the pieces of codimension
$k$ correspond to polymers with $k$ (facial) cycles.

A polymer $X$ with cycles lies in the boundary of each $\BP_R(T)$ for which
$T$ is a spanning tree of the graph $G(X)$.  Each such $\BP_R(T)$ will contribute its
own paramaterization to $X$.  Note, however, that some of the spanning trees may be
unrealizable by unit disks (e.g.\ the star inside a 6-wheel); we just regard
them as $\BP_R(T)$ of zero volume.

We can construct a model for the parameter space of all polymers of size $n$
and disk radii $R$ by taking a copy of $\BP_R(T)$ for each possible combinatorial
type of tree, and identifying boundaries as above. Note that the identification maps
are in general analytic maps on the angles: in a polygon with $k$ vertices whose
edges have fixed lengths $r_1,\dots,r_k$, any two consecutive angles are determined
analytically by the remaining $k-3$ angles.

\subsection{Perturbations}

A polygon $P_m$ with $m$ edges is determined up to rigid motion by
$m{-}1$ consecutive edge lengths $s_1,\dots,s_{m-1}$ and the $m{-}2$ consecutive
interior angles $\phi_1,\dots,\phi_{m-2}$ with $\phi_i$ between edges $s_i$ 
and $s_{i+1}$.

The space of perturbations of the angles of an $m$-gon $P_m$ which preserve
the edge lengths is $m{-}3$-dimensional, and is generated by ``local''
perturbations which change only four consecutive angles.  Here by perturbation
we mean the derivative at $0$ of a smooth one-parameter path in the space
of $m$-gons with the same edge lengths as $P_m$.  Such a perturbation is
determined by the derivatives of the angles with respect to the parameter
$t$ along the path.  We define $\pati$ to be the infinitesimal perturbation
of the angles of $P_m$, preserving the edge lengths, for which
$\frac{\pa\phi_j}{\pa t_i}=0$ unless $j$ is one of $i-1,i,i+1,i+2$ (indices
chosen cyclically) and $\frac{\pa\phi_i}{\pa t_i}=1$. See Figure~\ref{fig:bp4}.

\begin{figure}[htbp]
\epsfxsize270pt
$$\epsfbox{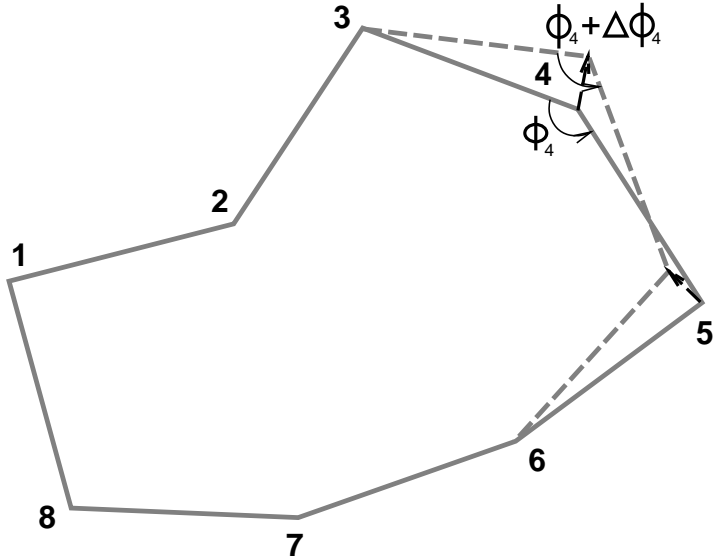}$$
\caption{Local perturbation of vertex 4 of an octagon}\label{fig:bp4}
\end{figure}

For generic $P$, the $\pati$ for $i=1,2,\dots,m-3$ generate all edge-length
preserving perturbations of $P$.  These $\pati$ are useful because they
provide a local infinitesimal coordinate charts for the boundaries of the
various sets $\BP_R(T)$ which share the same cycle $P_m$.

For example, suppose that $\BP_R(T)$ for some $T$ is parametrized by angles 
$\theta_1,\dots,\theta_{n-1}$, and we are on a part of the boundary defined by a
cycle with interior angles $\phi_1,\dots,\phi_m$ (so the $\phi$'s are
differences of the $\theta$'s). 
The infinitesimal (signed) volume of the part of the 
boundary swept out by the perturbations 
$\pati,$ for $i=1,\dots,m-3$
is
\begin{equation}\label{detform}
(d\phi_1\wedge\dots \wedge d\phi_{m-3})
\left(\frac{\pa}{\pa t_1},\dots,\frac{\pa}{\pa t_{m-3}}\right)
=\det\left(\frac{\pa\phi_i}{\pa t_j}\right)_{1\leq i,j\leq m-3}.
\end{equation}
That is to say,
if $P(t_1,\dots,t_{m-3})$ for $(t_1,\dots,t_{m-3})\in[0,\eps]^{m-3}$
is a perturbation of $P_m$, with $P(0,\dots,0)=P_m$ and with 
$\frac{\pa P(0,\dots,0)}{\pa t_i}$
having the above properties, then the signed volume swept out by all these
perturbations, divided by $\eps^{m-3}$, is given by (\ref{detform}).

We also need to consider perturbations of $P_m$ which change the edge lengths.
Let $\frac{\pa}{\pa S_i}$ be the perturbation in which two consecutive edge lengths
increase and the rest remain unchanged, and only three angles change:
\begin{eqnarray*}
\frac{\pa s_i}{\pa S_i}&\!\!=&\!\!\frac{\pa s_{i+1}}{\pa S_i}=1 \text{ and
$\frac{\pa s_j}{\pa S_i}=0$ for other $j$, and}\\ 
\frac{\pa\phi_j}{\pa S_i}&\!\!=&\!\!0 \text{ unless $j=i-1,i$ or $i+1$.}
\end{eqnarray*} This is obtained
by moving only vertex $i$.  Note that $\frac{\pa}{\pa S_i}$ for
$i=1,\dots,n$ generate all infinitesimal perturbations of the edge lengths.
With the above $\pati$, these $\frac{\pa}{\pa S_j}$ generate all motions of $P_m$.

\subsection{Volumes}

Here we determine how the volume of $\BP_R(T)$ changes when one of the radii is increased.

Let $X$ be an order-$n$ polymer in the boundary of $\BP_R(T)$. We assume that
$X$ is in a codimension-$1$ part of the boundary, that is, it has only one cycle,
$C$, with vertices $v_1,\dots,v_m$ in counterclockwise order.
Let $\phi_1,\dots,\phi_m$ be the corresponding interior angles of $C$.
Let us assume that the tree $T$ contains all edges of $C$ except the
edge between $v_{m-1}$ and $v_{m}$, so that angles $\phi_1,\dots,\phi_{m-2}$
parametrize the polymers in $\BP_R(T)$ close to $X$.

Recall that each edge of $T$ is directed away from the root.
Some of these are directed clockwise around $C$ and some counterclockwise.
Let $v_i$ be the vertex of $C$ closest to the root. 
The edges of $C$ are oriented clockwise around $C$ to the left of $v_i$ and counterclockwise
to the right of $v_i$. 
Let $\theta_{v_j}$ be the angle of the edge of $T$ whose head is at $v_j$.

\begin{lemma}\label{lemma:wedge}
\begin{equation}\label{changeform}
d\theta_{v_1}\wedge\dots\wedge\widehat{d\theta_{v_i}} \wedge\dots\wedge 
d\theta_{v_m}=
(-1)^{m-2}d\phi_1\wedge d\phi_2\wedge\dots \wedge d\phi_{m-2}\wedge d\theta_{v_m}.
\end{equation}
\end{lemma}

\begin{proof}
We can rewrite the left-hand side in terms of the $d\phi_j$ as follows. 
To the right of $v_i$,
we have $\phi_j=\pi-(\theta_{v_{j+1}}-\theta_{v_j}).$ To the left of 
$v_i$, we have $\phi_k=\pi+\theta_{v_{k-1}}-\theta_{v_k}$. Also,
$\phi_i=\theta_{v_{i-1}}-\theta_{v_{i+1}}$. 
Replace successively $d\theta_{v_k}$ by $-d\phi_k$ for $k=1,2,...,i-1$.
Then replace successively $d\theta_{v_\ell}$ by $-d\phi_{\ell-1}$ for
$\ell=i+1,i+2,\dots,m-1$. There are precisely $m{-}2$ minus signs.
\end{proof} 

The volume form on $\BP_R(T)$ is $d\theta_{1}\wedge\dots\wedge d\theta_{n-1}$.
We can write this as $\alpha\wedge\beta$, where $\alpha$ is the form on the
left-hand side of (\ref{changeform}). The form $\beta$ involves
the edges which are not part of the cycle $C$.

If we fix the angles of a polymer $X$ in the boundary of
$\BP_R(T)$ and change the radii by a small amount,
from $R$ to $R'$, the boundary of $\BP_{R'}(T)$ moves in general away from $X$.
To compute the change in volume of $\BP_R(T)$, we integrate,
along the entire codimension-$1$ boundary of $\BP_R(T)$,
this displacement times the volume form of the boundary.
Let us consider a radius perturbation $\frac{\pa}{\pa S_i}$.
On the boundary we use the local infinitesimal coordinates defined by the
perturbations $\pati$ for $i=1,\dots,{m-3}$.
The angles $\theta$ which are not part of the cycle $C$ can be perturbed
independently; let $\frac{\pa}{\pa\theta_j}$ be a perturbation
of an angle $\theta_j$ not on $C$, which changes only this angle.
The angle $\theta_m$ can also be perturbed independently of the angles
$\phi_j$ in the cycle and other angles:
it just determines the orientation of the cycle. Let $\frac{\pa}{\pa\theta_m}$
be a perturbation of $\theta_m$.
Let $\frac{\pa}{\pa\Theta}$ be the product of the perturbations for
$\theta_m$ and the remaining angles not in $C$. 

The local volume element gained or lost by $\BP_R(T)$ is then the product
of $\beta(\frac{\pa}{\pa\Theta})$ and
$$
\omega=(d\phi_1\wedge\dots \wedge d\phi_{m-2})
\left(\frac{\pa}{\pa S_j},\frac{\pa}{\pa t_1},\dots,\frac{\pa}{\pa t_{m-3}}\right).
$$

The total volume change of $\BP_R(T)$ is the integral of $\omega$ times 
$\beta(\frac{\pa}{\pa\Theta})$ over the entire boundary of $\BP_R(T)$.
(More precisely, there is a corresponding form $\omega=\omega(C)$ for each 
codimension-$1$ piece of the boundary corresponding to the possible
cycles $C$ formed by $T$. The sum of the integrals of each form $\omega(C)
$ times $\beta(\frac{\pa}{\pa\Theta})$ over the corresponding part of the
boundary gives the total volume change.)

Now as $T$ ranges over the trees obtained by removing one edge from $C$,
the individual forms on the right-hand side of (\ref{changeform}) are obtained from 
$d\phi_1\wedge\dots\wedge d\phi_{m} $ by removing two consecutive $d\phi_i$'s,
and wedging the result with $d\theta_{v_m}$ (which, since it plays the role of
a global rotation of the cycle, is the same for each $i$).

To prove that the sum of the volumes of the $\BP_R(T)$ is constant,
it suffices now to observe the following:

\begin{lemma}\label{lemma:constant}
If $\phi_1+\dots+\phi_m$ is constant, then
$$
\sum_{i=1}^m d\phi_1\wedge\dots \wedge d\phi_{i-1}\wedge d\phi_{i+2}\wedge\dots\wedge d\phi_m=0
$$
(with cyclic indices, and where if $m$ is even we must put a $-$ sign in front of the last term
$d\phi_2\wedge\dots\wedge d\phi_{m-1}$).
\end{lemma}

\begin{proof}
Substitute $d\phi_m=-d\phi_1-\dots-d\phi_{m-1}$ in each term and simplify.
\end{proof}

It remains to show that this constant volume in fact takes the claimed value.

\begin{theorem}\label{thm:value}
For any radius vector $R$ of length $n$, the volume of the space
of branched polymers is $(n{-}1)!(2\pi)^{n-1}$.
\end{theorem}

\begin{proof} Choose $\ep>0$ very small and let $R$ be given
by $r_i = \ep^i$.  Let $X$ be a uniformly random configuration of disks
with these radii, forming some labeled tree $T$.  Suppose that for some $j<n$, disks $1$
through $j$ are connected.  Then we claim that with probability near 1,
disk $j{+}1$ touches one of disks 1 through $j$.  To see this, observe
that otherwise disk $j{+}1$ is connected to some previous disk $i$,
$1\le i \le j$, via a chain of (relatively) tiny disks whose indices
all exceed $j{+}1$.  Let disk $k$, $k > j{+}1$, be the one that
touches disk $j{+}1$; then the angle of the vector from the center of
disk $k$ to the center of disk $j{+}1$ is constrained to a small range,
else disk $j{+}1$ would overlap disk $i$.  It follows that $\BP_R(T)$
has lost almost an entire degree of freedom, thus has very small
volume; in other words, the tree $T$ is very unlikely.

Suppose, on the other hand, that for every $j$, disks 1 through $j$ are connected.
Then we may think of $X$ as having been built by adding touching disks
in index order, and since each is tiny compared to all previous disks,
there is almost a full range $2\pi$ of angles available to it without
danger of overlap.

It follows that as $\ep \to 0$ the volume of the space of polymers
with radius vector $R$ approaches the volume of the space of crossing
inductive trees, namely $(n{-}1)!(2\pi)^{n-1}$.  Since this volume does
not depend on $R$, we have equality.
\end{proof}

\subsection{Generalization to graphs}

Let $G$ be a graph on vertices $\{1,\dots,n\}$ whose edges are
equipped with positive real lengths $r_{ij}$.  A {\bf $G$-polymer}
is a configuration of points in the plane, also labeled by
$\{1,\dots,n\}$, such that:
\begin{enumerate}
\item point number 1 is at the origin;
\item for each edge $\{i,j\}$ of $G$, the distance $\rho(i,j)$
between points $i$ and $j$ is at least $r_{ij}$; and
\item the edges $\{i,j\}$ for which $\rho(i,j)=r_{ij}$
span $G$.
\end{enumerate}
We denote the set of $G$-polymers realizing
a given (spanning) tree $T$ by $BP_G(T)$.

Note that if $R = (r_1,\dots,r_n)$, and $G$ is the complete
graph $K_n$ with $r_{ij}=r_i+r_j$, then a $G$-polymer
is precisely the set of centers of the disks of a polymer
with radius vector $R$, in the sense of the previous sections.
The volume $V_G$ of the space of $G$-polymers is defined
as before by the angles made by the vectors from $i$ to $j$,
where $\{i,j\}$ is an edge for which $\rho(i,j)=r_{ij}$.

In fact, the proof of Lemmas~\ref{lemma:wedge}~and~\ref{lemma:constant}
extend without modification to show that $V_G$ does not depend on
the lengths $r_{ij}$ (even if they fail to satisfy the triangle
inequality), but only on the structure of $G$.  This leaves us
with the question of computing $V_G$ for a simple graph $G$.

To do this, we label the edges of $G$ arbitrarily as
$e_1,\dots,e_m$ and if $e_k=\{i,j\}$ we choose its edge-length $r_{ij}$
to be $\ep^k$ for $\ep>0$ and very small.  Then (since $\ep \le \frac12$),
for the volume of $BP_G(T)$ to be non-zero, there must not be an edge $e_k$ of
$G \setminus T$ such that $k$ is the lowest index of all edges in the cycle
made by adjoining $e_k$ to $T$.  If no such edge exists we say that $T$ is ``safe'';
and in that case, arguing as in the proof of Theorem~\ref{thm:value}, there
is almost no danger of violating condition (2) above in a random element
of $BP_G(T)$.  Thus the volume of the space of configurations in $BP_G(T)$ is nearly
the full $(2\pi)^{n-1}$.

It follows that the volume of the space of {\em all} $G$-polymers
is $\mu(G)(2\pi)^{n-1}$, where $\mu(G)$ is the number of safe spanning trees
of $G$.  Since $\mu(G)$ does not depend on the edge labeling, one might
suspect that it has a symmetric definition, and indeed it does.

\begin{lemma}\label{lemma:mu}
For any graph $G$, the number $\mu(G)$ of safe spanning trees of $G$
is equal to the absolute value of the sum over all spanning subgraphs $H$ of $G$,
of $(-1)^{|H|}$.
\end{lemma}

\begin{proof}
A simple inclusion-exclusion argument suffices.  Let us fix any numbering of the edges
of $G$ and, for each spanning tree $T$, let $B(T)$ be the set of ``bad'' edges of
$G \setminus T$, that is, edges which boast the lowest index of any edge in the cycle
formed with $T$.  Associate to each spanning graph $H$ the spanning tree $T(H)$ obtained
by repeatedly removing the lowest-indexed edge from each cycle.  Then for $n$ odd, a
spanning tree $T$ with set $B(T)$ of bad edges is counted once positively for each
even subset of $B(T)$ and once negatively for each odd subset; and vice-versa for
$n$ even.  It follows that in the sum (which we denote by $\mu(G)$) $T$ has a net
count of 0 unless $B(T)$ is empty, in which case it counts once positively (for $n$
odd) or negatively ($n$ even).  But the trees for which $B(T)$ is empty are exactly
the safe trees.
\end{proof}

We conclude:

\begin{theorem}\label{thm:graph}
The volume of the space of $G$-polymers in the plane is $\mu(G)(2\pi)^{n-1}$.
\end{theorem}

Comparing with Theorem~\ref{thm:value}, we have indirectly shown that $\mu(K_n)= (-1)^{n-1}(n{-}1)!$.
In general $\mu(G) = |{\mathcal T}_G(0,1)|$ where ${\mathcal T}_G$ is the 
Tutte polynomial of $G$ (see e.g.\ \cite{B,C,T}).  We note also that $\mu(G)$ plays the
role of Brydges and Imbrie's function ``$J_C$'' in the dimension-2 case.

The computation of ${\mathcal T}_G(0,1)$, hence also of $\mu(G)$, is unfortunately
\#P-hard for general $G$ \cite{JVW}.  The point (0,1) is not, however, in the region of the
plane in which Goldberg and Jerrum \cite{GJ} have recently shown the Tutte polynomial to be hard even
to approximate.  Thus, there is some hope that a ``fully polynomial randomized approximation
scheme'' can be found for $\mu(G)$.

We conclude this section with a new solution of a notoriously difficult puzzle, which appears
as an exercise in \cite{Sp}, derived from Rayleigh's investigation (see \cite{W}) of ``random flight.''
The exercise calls for proving the corollary below by developing the Fourier analysis of spherically
symmetric functions, then deriving a certain identity involving Bessel functions.  Curiously, it is
(we believe) the only mention of {\em continuous} random walk in Spitzer's entire book.

\begin{corollary}\label{cor:walk}
Let $W$ be an $n$-step random walk in $\R^2$, each step being an independent
uniformly random unit vector.  Then the probability that $W$ ends within distance 1
of its starting point is $1/(n{+}1)$.
\end{corollary}

\begin{proof}
The volume of the space of such walks, beginning from the origin, is of course
$(2\pi)^n$.  If the walk does {\em not} terminate inside the unit disk at the origin,
it is in effect a $C_{n+1}$-polymer, where $C_{n+1}$ is the $n{+}1$ cycle
in which vertex $i$ is adjacent to vertex $i{+}1$, modulo $n{+}1$.
Since $\mu(C_{n+1})=|1-(n{+}1)|=n$, the volume of the space of $C_{n+1}$-polymers
is $n(2\pi)^n$.  Since the spanning tree with no edge between nodes 1 and $n{+}1$
is one of $n{+}1$ symmetric choices, the volume of the $C_{n+1}$-polymers which
correspond to non-returning random walks is $n(2\pi)^n/(n{+}1)$, and the result follows.
\end{proof}

\subsection{The bipartite case}

One special graph of interest is the complete bipartite graph $K_{m,n}$, representing
particles of two types, each particle interacting only with particles of the other type.
Note that in the hard-core model, phase transition has been proved in this situation \cite{R}---in
dimensions 2 and higher---but not for the complete graph.

Put $\mu_{m,n} = \mu(K_{m,n})$ and let $H(x,y)$ be the exponential generating function
for $\mu_{m,n}$, given by
$$
H(x,y)=\sum_{m=1}^\infty \sum_{n=1}^\infty \mu_{m,n}\frac{x^m}{m!} \frac{y^n}{n!}~.
$$

\begin{lemma}
$$
H(x,y) = \log(e^{-x} + e^{-y} - e^{-x-y})~.
$$
\end{lemma}

\begin{proof}
Observe that
$$
\mu_{m,n} = \sum (-1)^k \mu(m_1,n_1)\mu(m_2,n_2)\dots\mu(m_k,n_k)C(m_1,\dots,m_k,n_1,\dots,n_k)
$$
where the sum is over all partitions $m=m_1+\cdots+m_k$ and $n{-}1=n_1+\cdots+n_k$,
and $C$ is a combinatorial factor---the product of two multinomial coefficients
divided by appropriate factorials when some of the pairs $(m_i,n_i)$ are equal.

We claim that $dH(x,y)/dx = -1+ e^{-y-H(x,y)}$.
This is because $dH/dx$ in effect removes one vertex
from the left ($m$-vertex) side of $G$; this leaves
singleton vertices from the right side, counted by $e^y$,
and other components, counted by $e^{H(x,y)}$.
The $-$ signs are contributed by edges connecting these components to the
removed vertex, and the $-1$ summand comes about because there must be
at least one remaining component.

Solving the differential equation with initial condition $H(0,y)=0$
yields
$$
H(x,y)=\log(e^{-x} + e^{-y} - e^{-x-y}) = -x-y+\log(e^x+e^y-1)~.
$$
\end{proof}

The argument generalizes to the complete $k$-partite case, giving
$$
H(x_1,\dots,x_k) = -\sum_{i=1}^k x_i + \log\left(1-k+\sum_{i=1}^k e^{x_k}\right)~.
$$

\subsection{Construction}\label{sec:2-constr}

We now show inductively how to construct a uniformly random
branched polymer of order $n$ in the plane.

We begin with a unit disk centered at the origin.  Suppose we have constructed
a polymer of size $n{-}1$, $n>1$.  We choose a uniformly random disk
from among the $n{-}1$ we have so far, then choose a uniformly random boundary
point on that disk and start growing a new disk tangent to that point.
If a disk of radius $1$ fits at that point, this will define a polymer of size $n$.

Otherwise there is a radius $0<r<1$ at which a cycle forms with the new disk and some
other disks present. At this point our polymer $X$ is in the boundary of the
space $\BP_R(T)$, where $R=\{1,1,\dots,1,r\}$, and we need to choose some other tree
$T'$ for which $X$ is in the boundary of $\BP_R(T')$, and which has the property that
increasing $r$ (and leaving the angles fixed) will not cause the disks to overlap.
There will be at least one possible such $T'$ because the volume of $\BP_R(T)$
is decreasing as $r$ increases and so must be compensated by an increase in volume of
some $\BP_R(T')$.  We choose randomly among the $\BP_R(T')$ with increasing volume,
with probability proportional to the infinitesimal change in the volumes of the
$\BP_R(T')$'s as $r$ increases. This ensures that the volume lost to $\BP_R(T)$
as $r$ increases is distributed among the other $\BP_R(T')$ so as to maintain
the uniform measure. (In the language of Markov chains, this is the detailed
balance condition). 

Figure \ref{fig:bp2id} shows
snapshots of the construction of a random polymer, in the process of growing
its third and fourth disks;
Figure~\ref{fig:bp500} shows a polymer of order 500 generated by this method.

\begin{figure}[htbp]
\epsfxsize420pt
$$\epsfbox{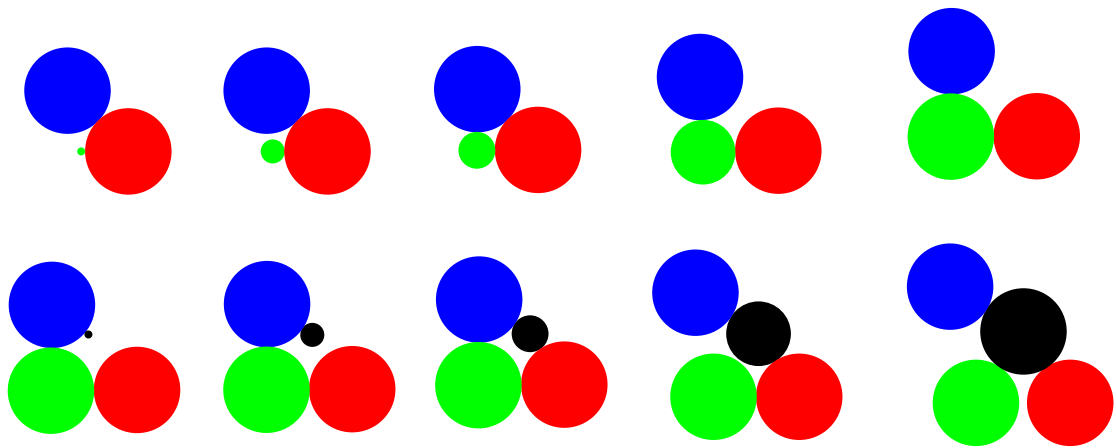}$$
\caption{A random planar branched polymer growing new disks}\label{fig:bp2id}
\end{figure}

\begin{figure}[htbp]
\epsfxsize450pt
$$\epsfbox{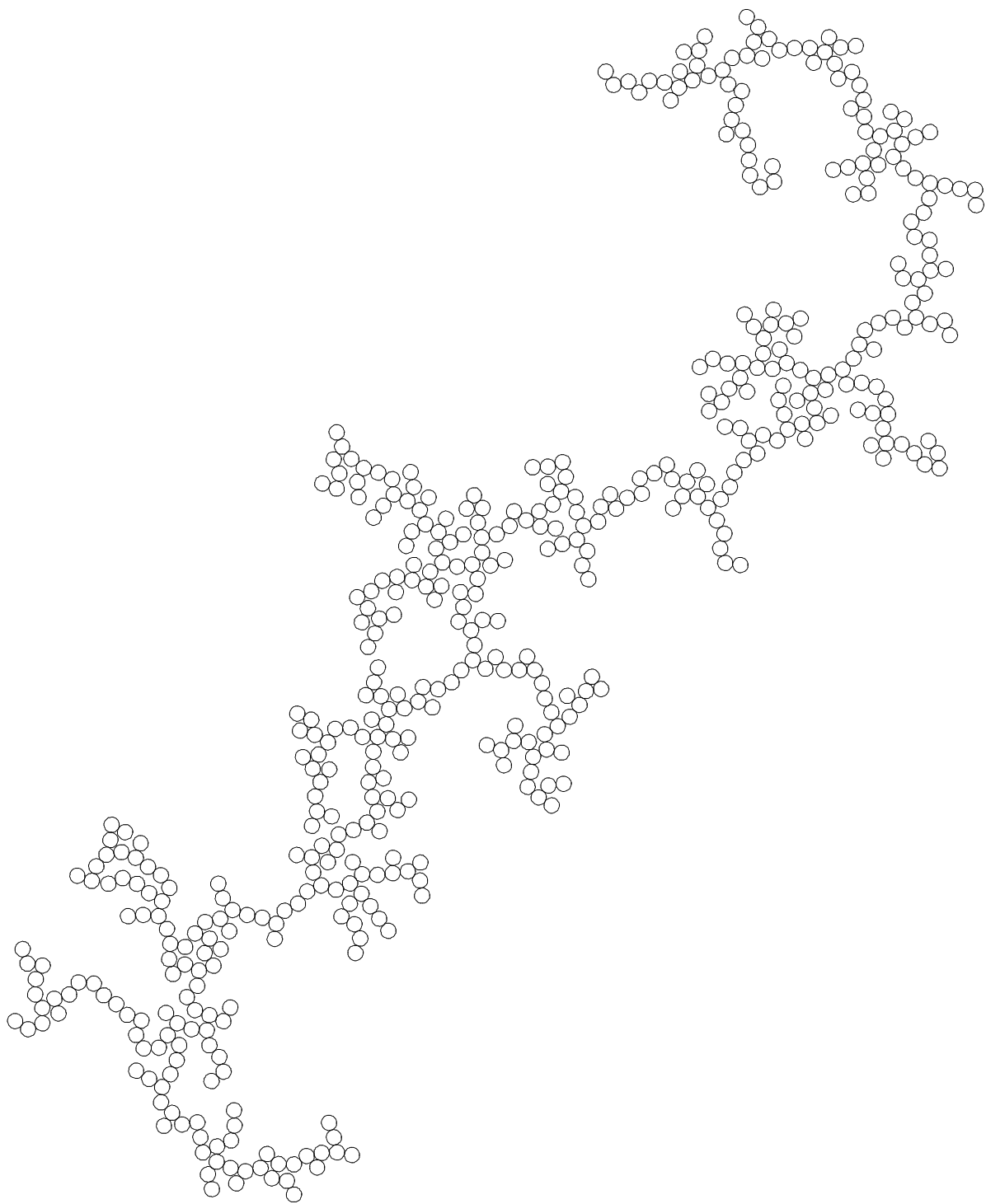}$$
\caption{A uniformly random two-dimensional branched polymer of 500 disks.}\label{fig:bp500}
\end{figure}

All of the above is easily generalized to produce uniformly random $G$-polymers
for any connected graph $G$ with specified edge-lengths (and in fact we will need
this construction later, when generating 3-dimensional polymers).  The vertices of
$G$ may be taken in any order $v_1,\dots,v_n$ having the property that the subgraph
$G_k$ induced by $v_1,\dots,v_k$ is connected for all $k$.  When a uniformly random
$G_{k-1}$-polymer has been constructed, a new point corresponding to vertex $v_k$
is added coincident to a point uniformly chosen from its neighborhood---in other
words, we start by assuming that the edges of $G_k$ incident to $v_k$ are infinitesimal
in length.  These edges are then grown to their specified sizes, breaking cycles when
they are formed in accordance with the rules above.

\section{The 3-dimensional Case}

\subsection{Volume invariance}

Branched polymers in 3-space share many of the features of planar branched polymers.
\cite{BI} showed that the volume of the configuration space of polymers in
3-space is $n^{n-1}(2\pi)^{n-1}$.  Whereas the planar configuration space volume
was independent of the radii of the balls, the same is not true in 3 dimensions.
However, there is an invariance inherited from the plane under a different
change of norm, which we now describe.

Let $G$ be a graph with $n$ vertices and edge weights $\beta_{ij}>0$.
A 3-dimensional $G$-polymer is a set of $n$ points 
$v_1,\dots,v_n\in\R^3$ such that for all $i,j$ we have
$$
\|v_i-v_j\|^2:=(v_i^1-v_j^1)^2+\beta_{ij}((v_i^2-v_j^2)^2+(v_i^3-v_j^3)^2)\geq1,
$$
with equality holding on a spanning tree of $G$.

When all $\beta_{ij}$ are $1$ this defines the standard branched polymer.
Note that if $\|v_i-v_j\|=1$ then $v_j$ is on the surface of an spheroid 
centered at $v_i$.  We measure the volume of the configuration space of
3-dimensional polymers using the normalized surface area of the
corresponding spheroids; we will see that this volume is independent of
the $\beta_{ij}$.

\subsection{One-dimensional projections}

Recall that the surface area measure of a sphere $S^2$, projected to a line
running through its center, projects to $2\pi$ times Lebesgue measure
on the image segment.  The same is true of the spheroid $\{v\in\R^3:\|v\|=1\}$
for any $\beta$, when projected to the $x$-axis, and it follows that for
purposes of computing the volume of the configuration space, we may assume
that the polymers are parametrized by the length of the projection
of each $v_i-v_j$ on the $x$-axis together with its angle to the 
positive $y$-axis
when projected onto the $yz$-plane.

Let $x_1,\dots,x_n$ be the projections of $v_1,\dots,v_n$ to the $x$-axis.
We suppose, after relabeling if necessary, that the $x_i$ are ordered
$x_1<x_2<\dots<x_n$.  If $v_i$ and $v_j$ are adjacent in the polymer then
$|x_i-x_j|\leq1$.  (See Figure~\ref{fig:proj}.)

\begin{figure}[htbp]
\epsfxsize280pt
$$\epsfbox{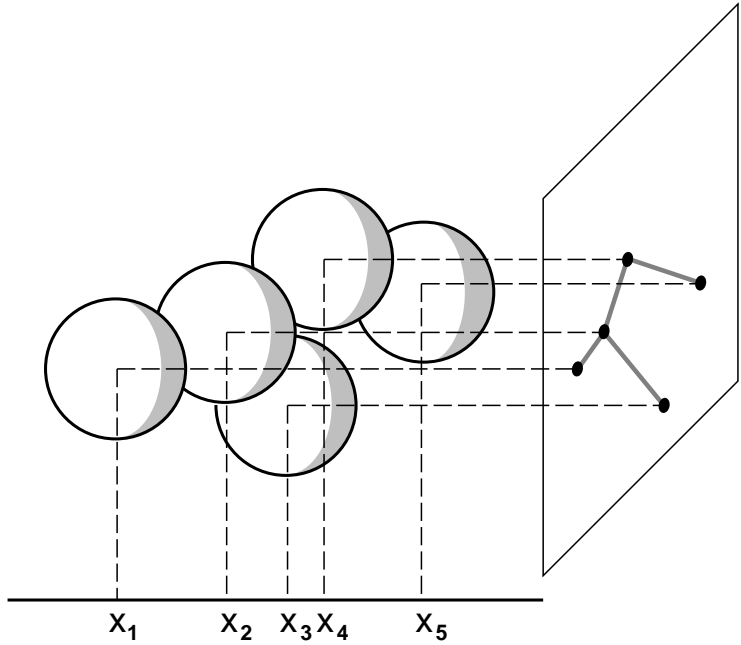}$$
\caption{A branched polymer projected onto the $x$-axis and $yz$-plane.}\label{fig:proj}
\end{figure}

\begin{lemma}\label{lemma:mult}
Fix a graph $G$ on $\{1,\dots,n\}$ with edge weights $\{\beta_{ij}\}$.
The $n{-}1$-dimensional volume of the set of $G$-polymers whose
centers project to $x_1<\cdots<x_n$ is an integer multiple of $(2\pi)^{n-1}$
and depends only on the set of pairs $i,j$ with $|x_j-x_i|>1$.
\end{lemma}

\begin{proof}
In any such polymer, the distance between the $yz$-plane projections
of each pair $i,j$ of adjacent centers is some fixed $r_{ij}$ depending
only on $|x_i-x_j|$ and $\beta_{ij}$.  For non-adjacent centers, this
distance is at least some $r_{ij}$ provided $|x_i-x_j|\leq 1$; otherwise
it is unconstrained.

It follows that if we let $H$ be the graph on vertices $\{1,\dots,n\}$ given
by $i\sim j$ iff $|x_i-x_j|\leq1$, and define $K$ to be the intersection of
(the edges of) $G$ and $H$ (with edge-lengths supplied by $G$), then by
Theorem~\ref{thm:graph} the desired volume is $\mu(K) \cdot (2\pi)^{n-1}$.
\end{proof}

\subsection{Complete graph}
Note that $H$ is a ``unit interval graph'' (see e.g.\ \cite{Ro})
defined by overlapping unit-length intervals, in this case with their
centers at the $x_i$.  In the case of standard polymers, where
$G$ is complete, $K$ is just $H$ with appropriate edge lengths assigned.
Thus $\mu(K)=\mu(H)$ and this is easy to compute: taking
an arbitrary total order on the edges and $\beta_{e_k} = \eps^{k}$ for $\eps$ small 
(where $e_k$ is the $k$th edge), the safe spanning trees of $H$
are only those which are ``inductive'' in the sense of the Introduction:
all paths from the root are increasing.  It follows that each vertex $j>1$
has as its parent some $i<j$ for which $x_i-x_j\leq 1$, thus
$$
\mu(H) = \prod_{j=2}^n \gamma(j)
$$
where $\gamma(j)$ is the number of $i<j$ for which $x_j-x_i\leq 1$.

It will be convenient temporarily to limit discussion to the space ${\BP}'_n$
of order-$n$ polymers for which $x_1=0$, i.e.\ those whose roots extend farthest to
the left along the $x$-axis.  Define the {\em type} $\sigma(X)$ of
a polymer $X$ in ${\BP}'_n$ to be the permutation $\sigma$ of $\{2,3,\dots,n\}$
for which $x'_{\sigma(2)} < x'_{\sigma(3)} < \cdots < x'_{\sigma(n)}$,
where $x'_i := x_i \mod 1$.  Then $\mu(H)$ depends only on $\sigma$ and
we may call it $\mu(\sigma)$.  The projections $x_1,\dots,x_n$ are uniquely
determined by $\sigma$ and the arbitrary subset $\{x'_1,\dots,x'_n\}$ of
$[0,1]$, thus:

\begin{corollary}\label{cor:type}
The volume of the space of branched polymers of type $\sigma$ in ${\BP}'_n$
is precisely $\mu(\sigma) (2\pi)^{n-1}$.
\end{corollary}

When $\sigma$ is the identity permutation $I$, all the $X_i$ are in [0,1], and
$\mu$ takes its maximum value $(n{-}1)!$. This is just the planar case in
disguise (although note that the three-dimensional volume of the polymers
of type $I$ incurs another factor of $n$ on account of the choice of $k$ for
which $x_k=0$).  On the other end of the scale, the minimum value $\mu(\sigma)=1$
is achieved in the spread-out case when $x_{i+2}-x_i>1$ for each $i=1,\dots,n{-}2$;
the number of such $\sigma$ is the ``Euler number'' $E_{n-1}$ (see e.g.\
Stanley \cite{St}).

Let $T_n$ be a uniformly random tree on the labels $\{1,\dots,n\}$,
with an independent uniformly random real length $u_{ij}$ in [0,1] assigned to
each edge $(i,j)$.  For each $j = 1,\dots, n$ let $a_j$ be the sum of the
lengths of the edges in the path from the root (vertex $1$) to $j$ in $T$; and let
$0=b_1 \le b_2 \le \cdots \le b_n$ be the $a_i$ taken in order. 
Let ${\bf B}$ be the (random) vector $\langle b_1,\dots,b_n \rangle$.

\begin{theorem}\label{thm:tree}
Let $X$ be a random branched polymer from ${\BP}'_n$, and $0=x_1 \le x_2 \le \cdots
\le x_n$ the projections of its centers onto the $x$-axis.  Then the random
vector $\langle 0, x_2, x_3, \dots, x_n \rangle$ is distributed as ${\bf B}$. 
\end{theorem}

\begin{proof}
Suppose first that a tree $T_n$ is fixed and that its ${\bf B}$-vector is
of type $\sigma$.  If we allow the edge-lengths of $T_n$ to vary, we find
that to maintain type $\sigma$ the quantities $b_2 \mod 1,\dots,b_n \mod 1$,
which are independent, uniformly random drawings from [0,1], must fall in
a particular order.  Thus the probability that the edge-length assignments
to any particular combinatorial tree $T$ will yield a ${\bf B}$-vector of any
fixed type $\sigma$ is either 0 or a constant independent of $T$ and $\sigma$.

In view of Corollary~\ref{cor:type}, it suffices then to show that the number of
labeled trees $T_n$  which contribute to type $\sigma$ is $\mu(\sigma)$,
but this is easy.  Given ${\bf B}$, the node of $T_n$ corresponding to $b_j$
must have as its parent (counting node 1 as root) some node corresponding
to an $i<j$ for which $b_j-b_i \le 1$.
\end{proof}

Theorem~\ref{thm:tree} says that the $x$-axis projections of a random
$X \in {\BP}'_n$ can be obtained by planting vertex 1 of $T_n$
at $x=0$ and stretching the tree to the right, letting the rest of
its nodes mark the projections.

\begin{figure}[htbp]
\epsfxsize290pt
$$\epsfbox{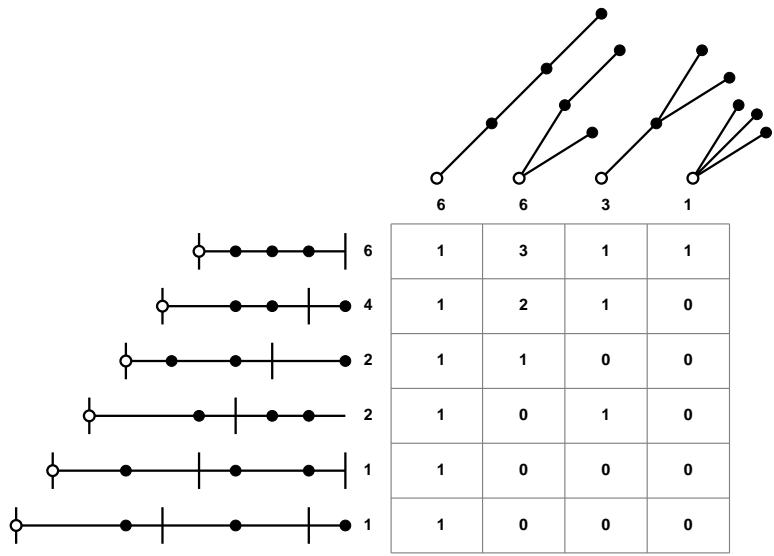}$$
\caption{The matrix of types and trees for $n=4$.}\label{fig:matrix}
\end{figure}

Figure~\ref{fig:matrix} illustrates the case $n=4$.  The rows are indexed
by types, presented as sample projections, each accompied by its
relative volume $\mu(\sigma)$.  The columns are indexed by trees,
each weighted by its number of distinct labelings (consistent with
root at 1).

Note that the theorem does {\em not} say that the tree structure of a
random 3-dimensional polymer is uniformly random; for example, no polymer
can have a node of degree greater than 12.  It does imply, however, when
combined with the proof of Theorem~\ref{lemma:mult}, that if the polymer is 
not made of spheres but of ellipsoids with widely ranging $y-z$ axes,
randomly assigned to labels, then indeed the tree structure approaches
uniformly random labeled trees.

From Theorem~\ref{thm:tree} we can incidentally deduce the not completely obvious fact that
the ``reverse'' vector $\langle 0, b_n-b_{n-1}, b_n-b_{n-2},\dots,b_n\}$
has the same distribution as ${\bf B}$.  For polymers, consequences 
of the theorem include the Brydges-Imbrie volume calculation and more:

\begin{theorem}\label{thm:3D}
The total volume of the space of 3-dimensional branched polymers of
order $n$ is $n^{n-1}(2\pi)^{n-1}$, and the expected diameter (combinatorial or
Euclidean) of a random such polymer grows as $n^{1/2}$.
\end{theorem}

\begin{proof} For the volume, we apply Cayley's theorem (to the effect that
the number of labeled $n$-node trees is $n^{n-2}$) and the fact that, since
relabeling of polymers preserves volume, the volume of ${\BP}'_n$ is just
$1/n$ times the volume of the whole space.

For the diameter we make use of Szekeres' Theorem (see \cite{RS,Sz}) saying
that the expected length of the longest path in a random tree on $n$ labels
is of order $\sqrt{n}$.  The expected length of the longest path from the
root in our edge-weighted tree $T_n$ must therefore also be of order $\sqrt{n}$,
and this is exactly the length of the projection of our random polymer on
the $x$-axis.  Since the space of polymers is independent of choice of
axes, the spatial diameter of a random polymer must also be of order $\sqrt{n}$.
\end{proof}

\subsection{Construction}

To construct a uniformly random three-dimensional branched polymer of order $n$,
we first select a uniformly random labeled tree $T$, then a set $\{x_1,\dots,x_n\}$
of projected centers on the $x$-axis.  We then build the $yz$-plane projection using
our 2-dimensional polymer construction; this yields the locations of the $n$
centers in 3-space, and it remains only to pick a root and translate it to the origin.

The tree $T$ on vertices $\{1,\dots,n\}$ can be selected from the $n^{n-2}$
possibilities by means of a Pr\"{u}fer code (see, e.g., \cite{LW}), which is itself
just a sequence of $n{-}2$ numbers between 1 and $n$.  The first entry of the code is the 
label of the vertex adjacent to the least-labeled leaf of $T$; that leaf is then deleted
and succeeding entries defined similarly.  The reverse process is also unique and easy.

The projections are defined by assigning independent uniformly random reals $u \in [0,1]$
to each edge of $T$, then letting $x_i$ be the length of the path from vertex $i$ to
vertex 1.  The unit-interval graph $H$ is defined as above on the tree-vertices,
namely by $i \sim j$ if $|x_j - x_i| \le 1$.  Edge-lengths are assigned to $H$ by
$\ell(i,j) = \sqrt{1-(x_j - x_i)^2}$ so that the spheres of the polymer corresponding to
tree vertices $i$ and $j$ are touching just when their centers lie at distance $\ell(i,j)$
when projected onto the $yz$-plane, and in any case lie at least that far apart.

From the argument above we know that given $x_1,\dots,x_n$, the $yz$-plane projections 
are exactly a uniformly random planar $H$-polymer, which we then select using the
methods of Section~\ref{sec:2-constr}.

Combining the $x$-axis and $yz$-plane projections gives us the centers of a
uniformly random branched polymer in 3-space (with spheres of diameter 1),
except that sphere number 1 is forced to have its center on the
$yz$-plane; we now choose a sphere uniformly at random to be the new root,
and translate the polymer so that this sphere's center is at the origin.

Figures \ref{fig:bp3d2}, \ref{fig:bp3d5} and \ref{fig:bp3d8} are snapshots, from
three angles, of a 3-dimensional branched polymer constructed as above.

\begin{figure}[htbp]
\epsfxsize460pt
$$\epsfbox{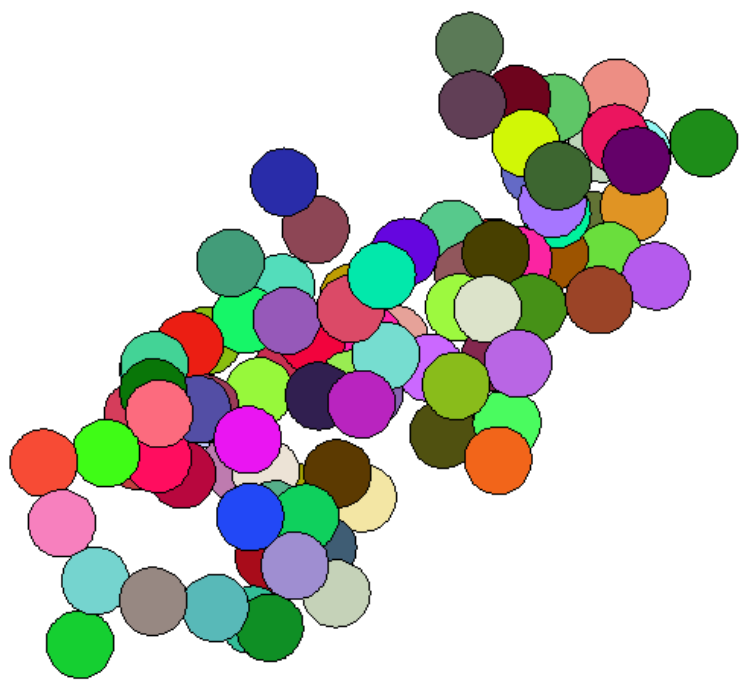}$$
\caption{A random branched polymer in 3-space}\label{fig:bp3d2}
\end{figure}

\begin{figure}[htbp]
\epsfxsize460pt
$$\epsfbox{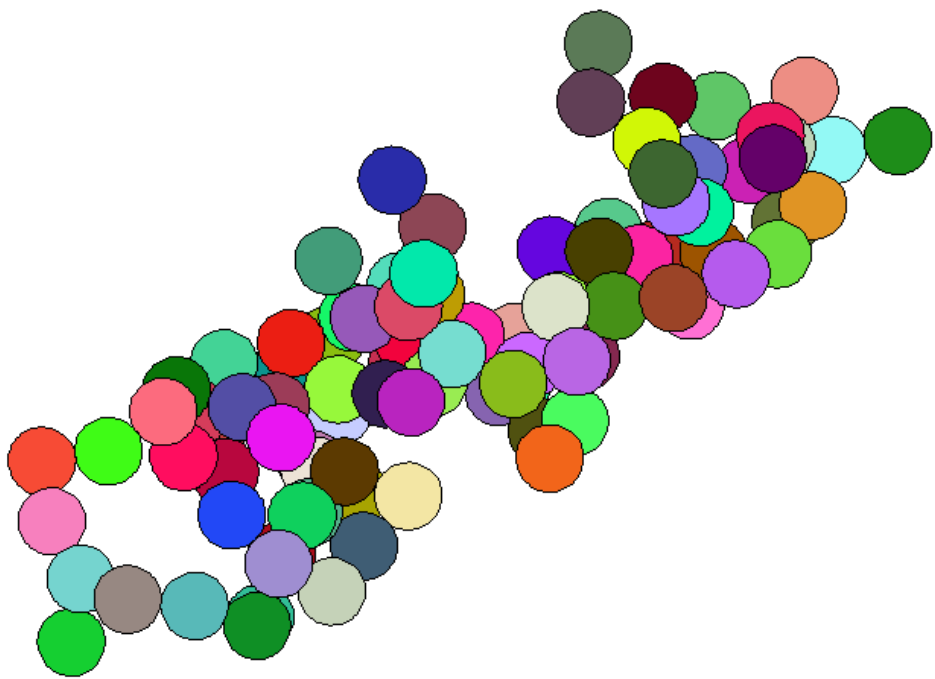}$$
\caption{The same polymer, slightly rotated}\label{fig:bp3d5}
\end{figure}

\old{\begin{figure}[htbp]
\epsfxsize460pt
$$\epsfbox{bp3d8.ps}$$
\caption{The same polymer, rotated a bit more}\label{fig:bp3d8}
\end{figure}
}

\section{Open problems}
\begin{enumerate}
\item
Is there a geometric interpretation of the local volume changes of the 
$\BP_R(T)$--which clearly depends on the shape of the cycle $C$?
This would lead to a possible natural geometrization of the space of polymers.
\item
What are the volumes of $\BP_R(T)$ for each $T$?
\item
What is the expected diameter (combinatorial or geometric) of a random
two-dimensional branched polymer?
\item
More generally, what do random polymers look like in the scaling limit,
in any fixed dimension?
\end{enumerate}

§


\begin{thebibliography}{1234}

\bibitem{B} N.L. Biggs, {\em Algebraic Graph Theory}, 2nd ed., Cambridge University Press (1993).

\bibitem{BI}
D.C. Brydges and J.Z. Imbrie, Branched polymers and dimension reduction, {\em Ann.\ Math.} {\bf 158} (2003), 1019--1039.

\bibitem{Bu} A. Bunde, S. Havlin, and M. Porto, Are branched polymers in the universality class of percolation?
{\em Phys.\ Rev.\ Lett.} {\bf 74} (1995), 2714--2716.

\bibitem{C} H.H. Crapo, The Tutte polynomial, {\em Aequationes Mathematicae} {\bf 3} (1969), 211--229.

\bibitem{F} F. Family, Real-space renormalisation group approach for linear and branched polymers,
{\em J. Phys.\ A: Math. Gen.} {\bf 13} (1980), L325--L334.

\bibitem{GJ} L.A. Goldberg and M. Jerrum (2006), Inapproximability of the Tutte polynomial, {\em ACM Symp. on the Theory of
Computing 2007}.

\bibitem{KS} D.J. Klein and W.A. Seitz, Self-similar self-avoiding structures: Models for polymers,
{\em Proc.\ Natl.\ Acad.\ Sci.\ USA} {\bf 80} \#10 (May 1983), 3125--3128. 

\bibitem{LW} J.H. van Lint and R.M. Wilson, {\em A Course in Combinatorics}, Cambridge (1992).

\bibitem{JVW} F. Jaeger, D.L. Vertigan, and D.J.A. Welsh (1990), On the computational complexity of the
Jones and Tutte polynomials, {\em Math.\ Proc.\ Cambridge Phil.\ Soc.} {\bf 108}, 35--53.

\bibitem{Lu} L.S. Lucena, J.M. Araujo, D.M. Tavares, L.R. da Silva, and C. Tsallis,
Ramified polymerization in dirty media: A new critical phenomenon, {\em Phys.\ Rev.\ Lett.} {\bf 72} (1994), 230--233.

\bibitem{RS} A. R\'{e}nyi and G. Szekeres, On the height of trees, {\em J. Austral.\ Math.\ Soc.} {\bf 7} (1967),
497--507.

\bibitem{Ro} F.S. Roberts, {\em Measurement Theory}, Addison-Wesley, Reading MA (1979).

\bibitem{R} D. Ruelle, Existence of a phase transition in a continuous classical system, {\em Phys.\ Rev.\ Lett.} {\bf 27}
(1971), 1040--1041.

\bibitem{Sp} F. Spitzer, {\em Principles of Random Walk}, Van Nostrand, Princeton (1964), p. 104.

\bibitem{St} R. Stanley, {\em Enumerative Combinatorics, Volume I}, Wadsworth \& Brooks/Cole, Monterrey CA (1986).

\bibitem{Sz} G. Szekeres, Distribution of labeled trees by diameter,
in {\em Combinatorial Mathematics X}, Springer-Verlag Lecture Notes in Mathematics \#1036 (1982).

\bibitem{T} W.T. Tutte, {\em Graph Theory}, Addison-Wesley, Reading MA (1984).

\bibitem{Va} C. Vanderzande, {\em Lattice Models of Polymers}, Cambridge U. Lecture Notes in Physics (1998).

\bibitem{Vu} D. Vuji\'{c}, Branched polymers on the two-dimensional square lattice with attractive surfaces,
{\em J. Stat.\ Phys} {\bf 95} \#3-4 (May 1999), 767--774.

\bibitem{W} G.N. Watson, {\em A Treatise on the Theory of Bessel Functions}, 2nd ed., Cambridge U. Press (1944), p. 419.

\end{thebibliography}
\end{document}